\documentclass[10pt]{article}
\usepackage{latexsym}
\usepackage{amssymb}
\usepackage{amsthm}
\usepackage{amsmath}
\usepackage{graphicx}
\usepackage[latin1]{inputenc}
\usepackage{hyperref}
\newtheorem{definition}{Definition}[section]
\newtheorem{theorem}{Theorem}[section]
\newtheorem{prop}[theorem]{Proposition}
\newtheorem{coro}[theorem]{Corollary}
\newtheorem{lemma}[theorem]{Lemma}
\newtheorem{remark}[theorem]{Remark}

\newcommand{\R}{\mathbb{R}}             
\newcommand{\N}{\mathbb{N}}             
\newcommand{\Z}{\mathbb{Z}}             %
\newcommand{\C}{\mathbb{C}}             
\renewcommand{\S}{\mathbb{S}}             

\author{Damien Gobin and Niky Kamran\footnote{Research supported by NSERC grant RGPIN 105490-2018.}}
\title{A Green's function for the source-free Maxwell equations on $AdS^5 \times \S^2 \times \S^3$}
\date{\today}

\begin{document}

\maketitle


\begin{abstract}
We compute a Green's function giving rise to the solution of the Cauchy problem for the source-free Maxwell's equations on a causal domain $\mathcal{D}$ contained in a
geodesically normal domain of the Lorentzian manifold $AdS^5 \times \S^2 \times \S^3$, where
$AdS^5$ denotes the simply connected $5$-dimensional anti-de-Sitter space-time. Our approach is to formulate the original Cauchy problem as an equivalent Cauchy problem for the Hodge Laplacian on $\mathcal{D}$ and to seek a solution in the form of a Fourier expansion in terms of the eigenforms of the Hodge Laplacian on $\S^3$. This gives rise to a sequence of inhomogeneous Cauchy problems governing the form-valued Fourier coefficients corresponding to the Fourier modes and involving operators related to the Hodge Laplacian on $AdS^5 \times \S^2$, which we solve explicitly by using Riesz distributions and the method of spherical means for differential forms. Finally we put together into the Fourier expansion on $\S^3$ the modes obtained by this procedure, producing a $2$-form on $\mathcal{D}\subset AdS^5 \times \S^2 \times \S^3$ which we show to be a solution of the original Cauchy problem for Maxwell's equations.

\vspace{0.5cm}
\end{abstract}

\tableofcontents

\section{Introduction and statement of the main result}



Our goal in this paper is to compute a Green's function for the source-free Maxwell equations on the $10$-dimensional symmetric Lorentzian manifold $({\mathcal M},g):=AdS^5 \times \S^2 \times \S^3$, where the metric $g$ is the product of the $5$-dimensional anti-de Sitter metric $g_{AdS^5}$ with the canonical round metrics $g_{\S^2}$ and $g_{\S^3}$ on $\S^2$ and $\S^3$, that is 
\begin{equation}\label{totalmetric}
g = g_{AdS^5} \oplus g_{\S^2} \oplus g_{\S^3}\,.
\end{equation}
In the above expression, $AdS^5$ stands for the complete simply connected Lorentzian $5$-manifold of constant sectional curvature $- \kappa$, with $\kappa > 0$ being fixed, and is thus diffeomorphic to $\R^5$. If $\theta = (\theta^1,\theta^2,\theta^3) \in [0,\pi] \times [0,\pi] \times \R/2 \pi \Z$ denote the
standard coordinates on the $3$-sphere, the metric on $AdS^5$ can be expressed globally as
\begin{equation}\label{metads5}
 g_{AdS^5} := \frac{- dt^2 + dx^2 + \cos(x)^2 g_{\S^3}}{\kappa \sin(x)^2}, 
\end{equation}
where $t \in \R$, $x \in (0,\pi/2 ]$ and where
\[ g_{\S^3} := (d\theta^1 )^2 + \sin^2(\theta^1)(d\theta^2)^2 + \sin^2(\theta^1) \sin^2(\theta^2)(d\theta^3)^2,\]
is the canonical round metric on $\S^3$.\\

While the problem of computing a Green's function for Maxwell's equations in this setting is of intrinsic geometric interest, it also carries a natural physical motivation coming from the fact that the symmetric manifold $({\mathcal M},g)$ is an exact solution of the field equations of Type IIA supergravity which enjoys natural T-duality properties (we refer to \cite{W} for a classification of all symmetric $AdS^n$ solutions of Type II supergravity and a discussion of their duality properties). Green's functions for hyperbolic differential operators are of course closely related to the solution of the initial value problem. More precisely, we shall see that for the specific class of symmetric Lorentzian manifolds $({\cal {M}},g)$ considered in this paper, the Green's function computed in our paper provides indeed the solution to a natural initial value problem for the source-free Maxwell equations. 

The approach we take to the computation of the Green's function makes use on the one hand of Riesz distributions and of generalizations to differential forms of the method of spherical means, and on the other of Fourier decompositions of differential forms in terms of eigenforms of the Hodge Laplacian. Our paper can thus be viewed as a natural continuation of earlier work by one of us in which Green's functions were constructed for the hyperbolic Hodge Laplacian on classes of Lorentzian 
symmetric spaces given by products of simply connected Lorentzian manifolds of constant sectional curvature with simply connected Riemannian surfaces of constant Gaussian curvature \cite {EK09}.

Before formulating our initial value problem, we first need to set up some basic notations. 
Thus we denote by $\Omega^k(\mathcal{M})$ the space of smooth $k$-forms on $\mathcal M$ and by $\Omega_0^k(\mathcal{M})$ the subspace of compactly supported smooth $k$-forms
on $\mathcal{M}$. We then write 
\[ \Omega_{0,d}^k(\mathcal{M}) = \{ \omega \in \Omega_0^k(\mathcal{M}) \, |  \, d \omega =0 \},\]
for the set of closed $k$-forms, and denote by $\Omega_{0,\delta}^k(\mathcal{M})$ the space of co-closed forms relative to the codifferential $\delta:=\delta_g$, 
\[ \Omega_{0,\delta}^k(\mathcal{M}) = \{ \omega \in \Omega_0^k(\mathcal{M}) \, |  \, \delta \omega =0 \}.\]
Unless otherwise stated, we write $*\omega$ for the Hodge dual $*_{g}\omega$ of a $k$-form $\omega$ relative to the metric $g$. We write $\Sigma^T$ for the space-like hypersurface in $\mathcal{M}$ defined by the level set $t = T$ and use the notation $\Sigma := \Sigma^0$.

We can now formulate the Cauchy problem for the Maxwell equations with initial data prescribed on $\Sigma$ using the
formalism of \cite{DL12,K01,P09}. To this effect, we first denote by $i : \Sigma \hookrightarrow \mathcal{M}$ the smooth embedding of $\Sigma$ in $\mathcal M$ by the inclusion map and introduce initial data given by $E \in \Omega_{0,\delta}^1(\Sigma)$ and $B \in \Omega_{0,d}^2(\Sigma)$. We are thus interested in solving
the following initial value problem for a $2$-form $F \in \Omega^2(\mathcal{M})$:
\begin{equation}\label{eqmax1}
 \left \{
\begin{array}{ccc}
    dF = 0 & \text{and} & \delta F = 0\,, \\
    -*_{\Sigma} i^{*} (* F) = E & \text{and} & - i^{*} F = B\,,\\
\end{array}
\right.
\end{equation}
where $*_{\Sigma}$ is the Hodge dual relative to the metric $i^{*}g$ induced by $g$ on the hypersurface $\Sigma$.\\
\noindent
In order to solve the above Cauchy problem we follow the approach of \cite{EK09} and restrict our analysis to a causal domain $\mathcal{D}$ contained in a geodesically
normal domain of $\mathcal{M}$ of $AdS^5$. This restriction will allow us to use arguments coming from the general theory of normally hyperbolic operators
on globally hyperbolic space-times, even though $AdS^5$ is not globally hyperbolic. For the sake of completeness let us now recall the definition of such a domain.

\begin{definition}[\cite{G88}]
 \begin{enumerate}
  \item If $U$ is the domain of validity for a normal coordinate system with origin $x \in U$, then $U$ is called a normal domain with
  respect to $x$. In that case there is for every $y \in U$ exactly one geodesic segment $[0,1] \ni t \mapsto c(t) \in U$ with $c(0) = x$, 
  $c(1)= y$ and $t$ an affine parameter.
  \item An open connected subset $\mathcal{N} \subset \mathcal{M}$ is called a geodesically normal domain if it is a normal domain with respect
  to each of its points.
  \item Let $\mathcal{N} \subset \mathcal{M}$ be a geodesically normal domain of the space-time $(\mathcal{M},g)$. An open subset $\mathcal{D}$ of $\mathcal{N}$ is called a
  causal domain in $\mathcal{N}$, if for every $x,y \in \mathcal{D}$ the set $J_\mathcal{N}^+ (x) \cap J_{\mathcal{N}}^- (y)$, where $J_{\mathcal{N}}^{\pm}$ are the causal
  future and past, is either empty or compact and contained in $\mathcal{N}$.
 \end{enumerate}
\end{definition}
\noindent
From now on, we work on a causal domain $\mathcal{D}$ contained in a geodesically normal domain of $\mathcal{M}$ and consider the Maxwell equations (\ref{eqmax1}) on
$\mathcal{D}$, i.e. with $F \in \Omega^2(\mathcal{D})$ and the set $\Sigma$ on which the initial values are defined being replaced by $\Sigma_{\mathcal{D}} = \Sigma \cap \mathcal{D}$. 

It should be noted that the general theory of the Cauchy problem for the wave equation and Maxwell's equations in curved space-time has been thoroughly investigated in the literature. A non-exhaustive list of important references includes the monographs \cite{Had, Ler, F75, G88} and the articles \cite{G66, C72, CB82, B02, B06, F04, F08,YZ14}.  Of direct relevance is a fundamental existence and uniqueness result for Maxwell's equations, obtained in \cite{G88}, Proposition 4.1, which states that if $\mathcal{D}$ is a causal domain of a Lorentzian manifold $(\mathcal{M},g)$, then there exists a unique solution $F \in \Omega^2(\mathcal{D})$ of the initial value problem (\ref{eqmax1}). 
 Our contribution in the present work can thus be viewed as an explicit construction of
this solution of the Maxwell equations in the particular case $\mathcal{M} = AdS^5 \times \S^2 \times \S^3$. This constitutes, as far as we know, one of the few concrete results
 of this kind for Maxwell's equations in higher-dimensional space-times. (We refer the reader to Appendix \ref{appersp} for some possible generalizations of the present work based on another strategy of proof.)
 
 We also note that in the framework of $4$-dimensional globally hyperbolic space-times a well-posedness result was obtained in \cite{K01} under
 some assumptions on the Cauchy surface $\Sigma$ and a more general result without any further assumptions on $\Sigma$ was proved in \cite{DL12}.
 A notable aspect of these results lies in the fact that they do not make use of a vector potential, that is a $1$-form $A$ such that $F=dA$, in order to solve the Maxwell equations, so
 that they do not require the topological condition $H^{2}_{dR}(\mathcal{M})=0$ to be satisfied. In contrast, the existence
 of such a vector potential is assumed for instance in \cite{D92,FP03}. In our work we will use the formalism and the strategy of \cite{DL12,G88}
 in order to avoid the introduction of such a potential $1$-form.\\

\newpage

In order to state our main theorem, we need to make a technical assumption about the structure of the causal domain $\mathcal{D}$ and to introduce some additional quantities. First, we assume that $\mathcal{D}$ factors as the product of $\S^3$ with a causal domain $\tilde{\mathcal{D}}$ contained in a geodesically normal domain of $AdS^5\times \S^2$, and likewise that the Cauchy hypersurface $\Sigma_{\mathcal{D}}$ factors as the product of $S^3$ with a Cauchy hypersurface $\Sigma_{\tilde{\mathcal{D}}}$ of ${\tilde{\mathcal{D}}}\subset AdS^5 \times \S^2$. Next, we perform a Fourier decomposition of the space of smooth $k$-forms on $\S^3$ in terms of 
eigenforms of the positive Hodge $k$-Laplacian $\Delta^{(k)} := - (d \delta + \delta d)$, namely, we write for $k \in \{0,1,2\}$,
\[ \Omega^k(\S^3) = \bigoplus_{\beta_k \in B_k} < \eta_{\beta_k}^{(k)} > \quad \text{with} \quad \Delta_{\S^3}^{(k)} \eta_{\beta_k}^{(k)} = \lambda_{\beta_k}^{(k)} \eta_{\beta_k}^{(k)}, \quad \forall \beta_k \in B_k,\]
where explicit formulas for the normalized eigenforms $\eta_{\beta_k}^{(k)}$ and the definition of the sets $B_k$ are given in Appendix \ref{apspeclap}. Thanks to this decomposition,
we can now expand any solution $F \in \Omega^2(\mathcal{D})$ of (\ref{eqmax1}) in terms of this Fourier basis of orthonormal eigenforms, that is, 
\[ F = \sum_{k=0}^2 \sum_{\beta_k \in B_k} F_{\beta_k} \wedge \eta_{\beta_k}^{(k)},\]
where $F_{\beta_k} \in \Omega^{2-k}(\tilde{\mathcal{D}})$. The object of our main result (see Theorem \ref{mainthm} below) is then to compute explicitly the coefficients $F_{\beta_k}$ of this Fourier series. This will be done after we prove in Sections \ref{secredwav} and
\ref{secdechod} the equivalence between Maxwell's equations (\ref{eqmax1}) and the sequence of Cauchy problems given on $\tilde{\mathcal{D}}$ by:
 \begin{equation*}
 \left \{
\begin{array}{ccc}
    &(\Delta_{\tilde{\mathcal{D}}}^{(2-k)} + \lambda_{\beta_k}^{(k)})F_{\beta_k} = 0& \\
    F_{\beta_k \, \, | \Sigma_{\tilde{\mathcal{D}}}} = \mathcal{F}_{0, \, \beta_k} & \text{and} & \nabla_{n} F_{\beta_k \, \, |\Sigma_{\tilde{\mathcal{D}}}} = \mathcal{F}_{1, \, \beta_k}\\
\end{array}, \quad k \in \{0,1,2 \},
\right.
\end{equation*}
where $F_{\beta_k} \in \Omega^{2-k}(AdS^5 \times \S^2)$ is defined on $\tilde{\mathcal{D}}$, where $\Sigma_{\tilde{\mathcal{D}}}$ is the restriction of $\Sigma_{\mathcal{D}}$ on $AdS^5 \times \S^2$ and where $\mathcal{F}_{0, \, \beta_k}$ and $\mathcal{F}_{1, \, \beta_k}$ are suitable initial data related to $E$ and $B$.\\
\noindent
Second, we need to introduce a particular smooth $(2-k)$-form $\hat{F}_{\beta_k}$ that matches the above initial data on the hypersurface $\Sigma$. More precisely, we let $\hat{F}_{\beta_k}$ be defined by
\begin{equation}\label{defhatF}
 \hat{F}_{\beta_k}(t,x) := \sum_{j=0}^{\infty} \sigma \left( \frac{t}{\epsilon_j} \right) t^j f_{j, \, \beta_k}(x), 
\end{equation}
where $\sigma : \R \rightarrow \R$ is a smooth function such that
\[ \sigma_{| [-1/2,1/2]} \equiv 1 \quad \text{with} \quad \sigma \equiv 0 \quad \text{outside} \quad [-1,1],\]
where $(\epsilon_j)_j$ is a suitable sequence taking values in the open interval $(0,1)$, and where the coefficients $f_{j, \, \beta_k}(x)$ are smooth $(2-k)$-forms that will be defined iteratively in (\ref{equ2}) below, ensuring that the preceding series can be differentiated term-wise and defines a smooth $(2-k)$-form. Moreover, 
\[ f_{0, \, \beta_k}(x) = \mathcal{F}_{0, \, \beta_k}(x) \quad \text{and} \quad f_{1,\, \beta_k}(x) = \mathcal{F}_{1, \, \beta_k}(x),\]
and the others $f_{j, \, \beta_k}$, $j\geq 2$ are defined thanks to the iterative procedure (\ref{equ2}).\\
\noindent
Next we introduce the following forms:
\begin{equation*}
  G_{+,\beta_k}(t,x) = \left \{
\begin{array}{ccc}
    (\Delta_{AdS^5 \times \S^2}^{(k)} + \gamma^{(k)})(\hat{F}_{\beta_k}) \quad &\text{if}& \quad x \in J_{\tilde{\mathcal{D}}}^+(\Sigma_{\tilde{\mathcal{D}}}) \\
    0 \quad &\text{if}& \quad x \in J_{\tilde{\mathcal{D}}}^-(\Sigma_{\tilde{\mathcal{D}}}) \\
\end{array},
\right.
\end{equation*}
and
\begin{equation*}
  G_{-, \, \beta_k}(t,x) = \left \{
\begin{array}{ccc}
    0 \quad &\text{if}& \quad x \in J_{\tilde{\mathcal{D}}}^+(\Sigma_{\tilde{\mathcal{D}}}) \\
    (\Delta_{AdS^5 \times \S^2}^{(k)} + \gamma^{(k)})(\hat{F}_{\beta_k}) \quad &\text{if}& \quad x \in J_{\tilde{\mathcal{D}}}^-(\Sigma_{\tilde{\mathcal{D}}}) \\
\end{array},
\right.
\end{equation*}
where $\gamma^{(k)} = \lambda_{\beta_{2-k}}^{(2-k)}$. We then finally define the forms $H_{\pm, \, \beta_k}$ by applying the Proposition \ref{soleqinhomo}
in the case $\omega = G_{\pm, \, \beta_k}$, i.e.
\begin{equation}\label{defHpm}
  H_{\pm, \, \beta_k}(x,y) := \int \overline{\omega_q(0,\lambda)} \omega_q(s,\lambda) \mathcal{R}_{x,D}^{2,\lambda+\gamma^{(k)}} (G_{\pm, \, \beta_k})(y)   d \mu_M (r) d \mu_{S}(s) d\rho_q(\lambda),
\end{equation}
where the functions $\omega_q$ are defined in Proposition \ref{specstudy} and $\mathcal{R}_{x,D}^{2,\lambda+\gamma^{(k)}}$ are the Riesz distributions introduced in
Definition \ref{defri}.

Using these different definitions we can now state our main theorem.

\begin{theorem}\label{mainthm}
 Assume that $\mathcal{D}$ is a causal domain contained in a geodesically normal domain of $(\mathcal{M},g)$, where
 $\mathcal{M}$ is the direct product $AdS^5 \times \S^2 \times \S^3$. The unique smooth
 solution $F \in \Omega^2(\mathcal{D})$ of the initial value problem for the Maxwell equations:
 \begin{equation}\label{caumax}
 \left \{
\begin{array}{ccc}
    dF = 0 & \text{and} & \delta F = 0 \\
    -*_{\Sigma} i^{*} (* F) = E & \text{and} & - i^{*} F = B\\
\end{array},
\right.
\end{equation}
where $E \in \Omega_{0,\delta}^1(\Sigma_{\mathcal{D}})$ and $B \in \Omega_{0,d}^2(\Sigma_{\mathcal{D}})$, is explicitly given by:
\[ F = \sum_{k=0}^2 \sum_{\beta_k \in B_k} F_{\beta_k} \wedge \eta_{\beta_k}^{(k)},\]
where
\begin{equation*}
  F_{\beta_k}(t,x) = \left \{
\begin{array}{ccc}
    F_{+, \, \beta_k}(t,x) &\text{if}& \quad x \in J_{\tilde{\mathcal{D}}}^+(\Sigma_{\tilde{\mathcal{D}}}) \\
    F_{-, \, \beta_k}(t,x) &\text{if}& \quad x \in J_{\tilde{\mathcal{D}}}^-(\Sigma_{\tilde{\mathcal{D}}}) \\
\end{array},
\right.
\end{equation*}
with
\[ F_{\pm, \, \beta_k}(t,x) = \hat{F}_{\beta_k}(t,x) - H_{\pm, \, \beta_k}(t,x),\]
where $\hat{F}_{\beta_k}(t,x)$ and $H_{\pm, \, \beta_k}(t,x)$ have been defined in (\ref{defhatF})-(\ref{defHpm}).
\end{theorem}

The strategy of our proof is the following. First, as mentioned previously, we follow the works \cite{DL12,G88,L10} in order to prove the equivalence between the Cauchy problem
(\ref{caumax}) and the Cauchy problem
 \begin{equation}\label{cauwav}
 \left \{
\begin{array}{ccc}
    &\Delta_{\mathcal{D}}^{(2)} (F) = 0& \\
    F_{| \Sigma_{\mathcal{D}}} = \mathcal{F}_0 & \text{and} & \nabla_{n} F_{|\Sigma_{\mathcal{D}}} = \mathcal{F}_1\\
\end{array},
\right.
\end{equation}
where $\Delta_{\mathcal{D}}^{(2)} = -(d \delta + \delta d)$ is the positive Hodge Laplacian acting on $2$-forms
and $\mathcal{F}_0$ and $\mathcal{F}_1$ are the restrictions to $\Sigma_{\mathcal{D}}$ of well chosen elements of $\Omega_0^2(\mathcal{D})$ depending on $E$
and $B$. We now study the Cauchy problem (\ref{cauwav}). Next, since $\mathcal{M} = AdS^5 \times \S^2 \times \S^3$ is a product manifold, the Hodge Laplacian $\Delta^{(2)}_{AdS^5 \times \mathbb{S}^2 \times \mathbb{S}^3}$ decomposes as follows: 
\[ \Delta^{(2)}_{AdS^5 \times \mathbb{S}^2 \times \mathbb{S}^3} = \begin{pmatrix}
\Delta^{(0)}_{AdS^5 \times \S^2} \oplus \Delta^{(2)}_{\mathbb{S}^3} & 0 & 0 \\
0 & \Delta^{(1)}_{AdS^5 \times \S^2} \oplus \Delta^{(1)}_{\mathbb{S}^3} & 0 \\
0 & 0 &\Delta^{(2)}_{AdS^5 \times \S^2} \oplus \Delta^{(0)}_{\mathbb{S}^3}
\end{pmatrix}.\]
where the blocks act on $\Omega^0(AdS^5 \times \S^2) \otimes \Omega^2(\S^3)$, $\Omega^1(AdS^5 \times \S^2) \otimes \Omega^1(\S^3)$ and
$\Omega^2(AdS^5 \times \S^2) \otimes \Omega^0(\S^3)$, respectively. We then perform a Fourier expansion of $F$ using the basis of eigenforms introduced above for the Hodge $k$-Laplacian on the $3$-sphere and obtain a sequence of Cauchy problems given, for $k \in \{0,1,2 \}$, by:
 \begin{equation}\label{cauwavpart}
 \left \{
\begin{array}{ccc}
    &(\Delta_{\tilde{\mathcal{D}}}^{(k)} + \lambda_{\beta_k}^{(2-k)})F_{\beta_k} = 0& \\
    F_{\beta_k \, \, | \Sigma_{\tilde{\mathcal{D}}}} = \mathcal{F}_{0, \, \beta_k} & \text{and} & \nabla_{n} F_{\beta_k \, \, |\Sigma_{\tilde{\mathcal{D}}}} = \mathcal{F}_{1, \, \beta_k}\\
\end{array},
\right.
\end{equation}
where as mentioned in introduction $\tilde{\mathcal{D}}$ is a copy of a subspace of $AdS^5 \times \S^2$ satisfying the the same properties as $\mathcal{D}$,
$(\lambda_{\beta_k}^{(k)})_{\beta_k}$ is the spectrum of the Hodge $k$-Laplacian on $\S^3$,
i.e. $\Delta_{\S^3}^{(k)}$, $\Sigma_{\tilde{\mathcal{D}}}$ is the Cauchy hypersurface of $AdS^5 \times \S^2$ corresponding to $\Sigma_{\mathcal{D}}$ for a fixed point
of $\S^3$, $F \in \Omega^{2-k}(AdS^5 \times \S^2)$ is defined on $\tilde{\mathcal{D}}$ and $\mathcal{F}_{0, \, \beta_k}$ and
$\mathcal{F}_{1, \, \beta_k}$ are the coefficients of the decomposition of the initial data $\mathcal{F}_{0}$ and $\mathcal{F}_{1}$ on the basis of eigenforms
$\eta_{\beta_k}^{(k)}$, respectively. The reason for making use of such a
decomposition lies in the fact that it reduces the original initial value problem to a sequence of initial value problems on the product manifold $AdS^5 \times \S^2$. This type of product fits into the framework studied in \cite{EK09}, where the inhomogeneous problem
\[ \Delta_{M \times S}^{(k)} F = \omega,\]
where $M$ is a Lorentzian manifold and $S$ is a Riemannian surface, is solved by using Riesz distributions and spherical means for differential forms. We shall see that the construction of \cite{EK09} can be readily modified if one adds a nonnegative constant to the Hodge Laplacian, meaning that we can solve the sequence of inhomogeneous problems given by
 \begin{equation}\label{inhomowav}
    (\Delta_{\tilde{\mathcal{D}}}^{(k)} + \lambda_{\beta_k}^{(2-k)})F_{\beta_k} = \omega, \quad k \in \{0,1,2\},
\end{equation}
corresponding to our Cauchy problem (\ref{cauwavpart}). Finally, following \cite{BGP07,G88}, we will show how to construct a solution of the Cauchy problem (\ref{cauwavpart}) as the sum of the explicit
solutions of the inhomogeneous problems introduced above and a smooth $2$-form (also explicitly constructed in the spirit of the $B$-series studied in \cite{G88})
satisfying the initial value conditions. As a last step we take the Fourier sum of the solutions obtained on each mode to obtain the complete solution, i.e. a solution of (\ref{cauwav}) which
is completely explicit, thereby giving an expression for a Green's function for the source-free Maxwell equations on $AdS^5 \times \S^2 \times \S^3$.

The paper is organized as follows. First, we show in Section \ref{secredwav} the equivalence between the Cauchy problem for the Maxwell equations and the Cauchy problem for
the Hodge Laplacian.
Next, in Section \ref{secdechod} we study the decomposition of the Laplacian and we reduce the Cauchy problem for the Hodge Laplacian on
$\mathcal{D} \subset \mathcal{M}$ into a countable family
of Cauchy problems for the Hodge Laplacian, up to a positive constant, on $\tilde{\mathcal{D}} \subset AdS^5 \times \S^2$. We then solve in Section \ref{secpbin}
the corresponding inhomogeneous problem using Riesz distributions and
spherical means. Finally, we construct the solutions of the sequence of Cauchy problems in Section \ref{secprepar} and we add them to obtain the complete solution of the Cauchy
problem for the Hodge Laplacian on the domain $\mathcal{D} \subset \mathcal{M}$ in Section \ref{secfinpre}.
There are also two appendices in this paper. Appendix \ref{apspeclap} deals with the spectral theory for the Hodge Laplacian on $\S^3$, while Appendix \ref{appersp}
suggests an alternative approach to the computation of a Green's function and some possible generalizations of our result to more general geometries which arise as solutions of supergravity, but which are of cohomogeneity one as opposed to being homogeneous.

\section{Proof of the main theorem}

\subsection{Reduction to a wave equation}\label{secredwav}

Our aim in this section is to prove the following Proposition.

\begin{prop}\label{eqmaxwav}
 Let $\mathcal{D}$ be a causal domain contained in a geodesically normal domain of $(\mathcal{M},g)$. The Cauchy problem for the Maxwell equations
\begin{equation}\label{caumax2}
 \left \{
\begin{array}{ccc}
    dF = 0 & \text{and} & \delta F = 0 \\
    -*_{\Sigma} i^{*} (* F) = E & \text{and} & - i^{*} F = B\\
\end{array},
\right.
\end{equation}
where $E \in \Omega_{0,\delta}^1(\Sigma_{\mathcal{D}})$ and $B \in \Omega_{0,d}^2(\Sigma_{\mathcal{D}})$, with $\Sigma_{\mathcal{D}} = \Sigma \cap \mathcal{D}$,
is then equivalent to the Cauchy problem for the Hodge Laplacian given by
 \begin{equation}\label{cauwav2}
 \left \{
\begin{array}{ccc}
    &\Delta_{\mathcal{D}}^{(2)} (F) = 0& \\
    F_{| \Sigma_{\mathcal{D}}} = \mathcal{F}_0 & \text{and} & \nabla_{n} F_{|\Sigma_{\mathcal{D}}} = \mathcal{F}_1\\
\end{array},
\right.
\end{equation}
where $\mathcal{F}_0$ and $\mathcal{F}_1$ are well chosen maps from $\Sigma_{\mathcal{D}}$ into $\Omega_0^2(\mathcal{D})$ depending on $E$ and $B$.
\end{prop}

\begin{proof}
 To do this we follow the idea of \cite{G88}, Proposition 4.1 p.283, \cite{DL12}, Proposition 2.1 and \cite{L10}, of which we recall the main lines for the sake of completeness.
 First, it is easy to check that if
 \[  dF = 0 \quad \text{and} \quad \delta F = 0,\]
 then
 \[ \Delta_{\mathcal{D}}^{(2)}(F) = - (d \delta + \delta d)(F) = 0.\]
 We now need to prove that we can always select suitable initial data for the wave equation so that a solution of (\ref{cauwav2}) is also a solution of (\ref{caumax2})
 as done in \cite{DL12,L10}. In the latter papers, the suitable Cauchy data are constructed by choosing maps $\mathcal{F}_0$ and $\mathcal{F}_1$ from
 $\Sigma_{\mathcal{D}}$ into $\Omega_0^2(\mathcal{\mathcal{D}})$ such that
 \[\mathcal{F}_{0 \, |V \cap \Sigma_{\mathcal{D}}} = n_0 E_j d\phi^0 \wedge d\phi^j - \frac{1}{2} B_{ij} d\phi^i \wedge d\phi^j,\]
 and
 \[\mathcal{F}_{1 \, |V \cap \Sigma_{\mathcal{D}}} = n^0 \nabla_i \mathcal{F}_{0j} d\phi^i \wedge d\phi^j - n_0 g^{ij} \nabla_i \mathcal{F}_{jk} d\phi^0 \wedge d\phi^k,\]
where $V$ is a coordinate patch of $\mathcal{D}$ that intersects $\Sigma_{\mathcal{D}}$ in a non empty set endowed with a local chart $\phi$ and $n_{\mu}$ is the unit normal vector
to $\Sigma_{\mathcal{D}}$. Concerning the initial data we refer the reader to the proof of Proposition 2.1 of \cite{DL12} and to \cite{L10} for additional details.
Let us now prove that a solution $F$ of (\ref{cauwav2})
is also a solution (\ref{caumax2}). Assume that
\[ \Delta_{\mathcal{D}}^{(2)}(F) = 0.\]
Since, $[ \Delta, \delta] =  [\Delta, d] = 0$, we can deduce from the previous equality that
\[ \Delta_{\mathcal{D}} ( \delta F) = 0 \quad \text{and} \quad \Delta_{\mathcal{D}}( d F) = 0.\]
We also have, thanks to our choice of initial data, that
\[ dF_{| V \cap \Sigma_{\mathcal{D}}} = 0 \quad \text{and} \quad \nabla_{n} dF_{| V \cap \Sigma_{\mathcal{D}}} = 0.\]
The quantity $dF$ is then a solution of the problem
 \begin{equation*}
 \left \{
\begin{array}{ccc}
    &\Delta_{\mathcal{D}}(dF) = 0& \\
    dF_{| \Sigma_{\mathcal{D}}} = 0 & \text{and} & \nabla_{n} dF_{|\Sigma_{\mathcal{D}}} = 0\\
\end{array}.
\right.
\end{equation*}
Since we restricted our analysis to a causal domain $\mathcal{D}$ contained in a geodesically normal domain of $\mathcal{M}$,
our framework is exactly the one used in the Proposition 4.1 p. 201 of \cite{G88}. We can thus
conclude that $dF = 0$ (see \cite{G88}, Proposition 4.1 p.283). Using the same argument
we can prove in the same way that $\delta F = 0$. We have thus proved that a solution $F$ of (\ref{cauwav2}) is actually also a solution
of (\ref{caumax2}).
 \end{proof}

\subsection{Decomposition of the Hodge Laplacian and sequence of Cauchy problems}\label{secdechod}

The purpose of this section consists in using the decomposition of the Hodge Laplacian on the product manifold $\mathcal{M} = AdS^5 \times \S^2 \times \S^3$ in order to
simplify the Cauchy problem (\ref{cauwav2}). More precisely, we shall use the eigenmodes on the $3$-sphere to transform this complete Cauchy problem into a 
sequence of Cauchy problems indexed by the spectrum of the Hodge Laplacian. Our motivation for reducing the complete Cauchy problem into Cauchy
problems on $AdS^5 \times \S^2$ lies in the fact that this framework was studied in \cite{EK09} and we want to use this work
in order to construct explicitly solutions of these Cauchy problems.

First, let us note that, thanks to the product structure of the manifold $\mathcal{M} = AdS^5 \times \S^2 \times \S^3$, the Hodge Laplacian acting on $\Omega^2(\mathcal{M})$ can be decomposed in a convenient way. Indeed, it follows from the product structure of ${\cal M}=AdS^5 \times \mathbb{S}^2 \times \mathbb{S}^3$ that this operator can be expressed in matrix form as
\[ \Delta^{(2)}_{AdS^5 \times \mathbb{S}^2 \times \mathbb{S}^3} = \begin{pmatrix}
\Delta^{(0)}_{AdS^5 \times \S^2} \oplus \Delta^{(2)}_{\mathbb{S}^3} & 0 & 0 \\
0 & \Delta^{(1)}_{AdS^5 \times \S^2} \oplus \Delta^{(1)}_{\mathbb{S}^3} & 0 \\
0 & 0 &\Delta^{(2)}_{AdS^5 \times \S^2} \oplus \Delta^{(0)}_{\mathbb{S}^3}
\end{pmatrix},\]
where the blocks act on $\Omega^0(AdS^5 \times \S^2) \otimes \Omega^2(\S^3)$, $\Omega^1(AdS^5 \times \S^2) \otimes \Omega^1(\S^3)$ and
$\Omega^2(AdS^5 \times \S^2) \otimes \Omega^0(\S^3)$, respectively.\\
\noindent
Let us now introduce the spectrum and the eigenforms of the Hodge $k$-Laplacians $\Delta_{\S^3}^{(k)}$. We refer the reader to Appendix \ref{apspeclap} for more details on this point and also 
to the references \cite{BHQR16,KMHR00,L04}. Thanks to these works we can decompose for $k=0,1,2$ the spaces of $k$-forms on the $3$-sphere without multiplicity
in the following way:
\[ \Omega^0(\S^3) = \bigoplus_{\beta_0 \in B_0} < \eta_{\beta_0}^{(0)} > \quad \text{where} \quad \Delta_{\S^3}^{(0)} \eta_{\beta_0}^{(0)} = \lambda_{\beta_0}^{(0)} \eta_{\beta_0}^{(0)}, \quad \forall \beta_0 \in B_0,\]
\[ \Omega^1(\S^3) = \bigoplus_{\beta_1 \in B_1} < \eta_{\beta_1}^{(1)} > \quad \text{where} \quad \Delta_{\S^3}^{(1)} \eta_{\beta_1}^{(1)} = \lambda_{\beta_1}^{(1)} \eta_{\beta_1}^{(1)}, \quad \forall \beta_1 \in B_1,\]
and
\[ \Omega^2(\S^3) = \bigoplus_{\beta_2 \in B_2} < \eta_{\beta_2}^{(2)} > \quad \text{where} \quad \Delta_{\S^3}^{(2)} \eta_{\beta_2}^{(2)} = \lambda_{\beta_2}^{(2)} \eta_{\beta_2}^{(2)}, \quad \forall \beta_2 \in B_2.\]
Note that the spectral theory of the Hodge Laplacian on the space of $3$-forms on $\S^3$, $\Omega^3(\S^3)$, is completely equivalent to the spectral theory on
$\Omega^0(\S^3)$ thanks to the Hodge duality. We also emphasize that the previous decompositions
are given without mutliplicity, meaning that an eigenvalue $\lambda_{\beta_k}^{(k)}$, $k \in \{0,1,2\}$, can be repeated.

Thus, restricted to a triplet $(\eta_{\beta_2}^{(2)},\eta_{\beta_1}^{(1)},\eta_{\beta_0}^{(0)})$, the operator $\Delta_{\mathcal{M}}^{(2)}$ is given by
\[ \Delta^{(2)}_{AdS^5 \times \mathbb{S}^2 \times \mathbb{S}^3} = \begin{pmatrix}
\Delta^{(0)}_{AdS^5 \times \S^2} + \lambda_{\beta_2}^{(2)} & 0 & 0 \\
0 & \Delta^{(1)}_{AdS^5 \times \S^2} +\lambda_{\beta_1}^{(1)} & 0 \\
0 & 0 &\Delta^{(2)}_{AdS^5 \times \S^2} + \lambda_{\beta_0}^{(0)}
\end{pmatrix}.\]
We remark that on the first block we can use the eigenmodes on the $2$-sphere $\S^2$ in the same way in order to obtain the scalar operator
 $\Delta^{(0)}_{AdS^5} + \mu + \lambda_{\beta_2}^{(2)}$, where $\mu$ is an eigenvalue of $\Delta^{(0)}_{\S^2}$, as was
 studied in \cite{EK10}, leading to a representation of the solution.

\noindent
Using these eigenforms we then obtain a sequence of Cauchy problems indexed by the spectrum of the Hodge Laplacian on $\S^3$, given on
$\tilde{\mathcal{D}}\subset AdS^5 \times \S^2$ by:
 \begin{equation*}
 \left \{
\begin{array}{ccc}
    &(\Delta_{\tilde{\mathcal{D}}}^{(2-k)} + \lambda_{\beta_k}^{(k)})F_{\beta_k} = 0& \\
    F_{\beta_k \, \, | \Sigma_{\tilde{\mathcal{D}}}} = \mathcal{F}_{0, \, \beta_k} & \text{and} & \nabla_{n} F_{\beta_k \, \, |\Sigma_{\tilde{\mathcal{D}}}} = \mathcal{F}_{1, \, \beta_k}\\
\end{array}, \quad k \in \{0,1,2 \},
\right.
\end{equation*}
where $F_{\beta_k} \in \Omega^{2-k}(AdS^5 \times \S^2)$ is defined on $\tilde{\mathcal{D}}$, $\Sigma_{\tilde{\mathcal{D}}}$ is the restriction of $\Sigma_{\mathcal{D}}$ on
$AdS^5 \times \S^2$ and $\mathcal{F}_{0, \, \beta_k}$ and $\mathcal{F}_{1, \, \beta_k}$ are the corresponding initial data.
To keep the notation simple, in the next three sections we will omit some subscripts and we thus rewrite these Cauchy problems as
 \begin{equation}\label{cauwavpart2}
 \left \{
\begin{array}{ccc}
    &(\Delta_{\tilde{\mathcal{D}}}^{(k)} + \gamma^{(k)})F_k = 0& \\
    F_{k \, \, | \Sigma_{\tilde{\mathcal{D}}}} = \mathcal{F}_{0, \, k} & \text{and} & \nabla_{n} F_{k \, \, |\Sigma_{\tilde{\mathcal{D}}}} = \mathcal{F}_{1, \, k}\\
\end{array}, \quad k \in \{0,1,2 \},
\right \}
\end{equation}
where $\gamma^{(k)} := \lambda_{\beta_{2-k}}^{(2-k)}$ and $F_k := F_{\beta_k} \in \Omega^{k}(\tilde{\mathcal{D}})$ is defined on the causal domain $\tilde{\mathcal{D}}$
contained in a geodesically normal domain of $AdS^5 \times \S^2$. The goal is now to solve each of the Cauchy problems
(\ref{cauwavpart2}) on $\Omega^k(AdS^5 \times \S^2)$. As mentioned earlier, this fits precisely into the framework studied in \cite{EK09} in which Green's functions were computed for Hodge Laplacians on products of simply connected Lorentzian symmetric spaces of constant sectional curvature with simply connected Riemannian surfaces of constant Gaussian curvature. 
\subsection{Solution of the inhomogeneous problem}\label{secpbin}

In this section we recall the results obtained in \cite{EK09} and slightly modify them in order to make them applicable to the sequence of inhomogeneous Cauchy problems (\ref{inhomeq}) on $\Omega^k(AdS^5 \times \S^2)$ which will stem from the sequence of Cauchy problems
(\ref{cauwavpart2})  (see Section \ref{secprepar} below). Most of this material is a slight variant of what appears in \cite{EK09}, but we have nevertheless chosen to include it in order to set up some key notation that will be used in the remainder of our paper and to make the exposition sufficiently self-contained. In \cite{EK09}, the inhomogeneous problem
\begin{equation}\label{inhomo}
  \Delta_{M \times S}^{(k)} F = \omega,
\end{equation}
is considered, where $M$ is a simply connected Lorentzian manifold of constant sectional curvature $k_M$,
$S$ is a simply connected Riemannian surface of constant Gaussian curvature and $\omega \in \Omega^p(D) \otimes \Omega^q(S)$, and $D$ is a geodesically normal domain of $M$.

An explicit solution of (\ref{inhomo}) is obtained for this problem in \cite{EK09} by constructing an advanced Green's operator for $\Delta_{M \times S}^{(k)}$ and using Riesz distributions and spherical means.\\
\noindent
For our purpose of computing a Green's function for Maxwell's equations in $({\cal M},g)$, the system (\ref{cauwavpart2}) obtained in the previous section by Fourier decomposition on $\S^3$ will lead to an inhomogeneous equation of the form 
\begin{equation}\label{inhomonous}
  (\Delta_{M \times S}^{(k)} + \gamma^{(k)}) F = \omega,
\end{equation}
where $M$ stands for $AdS^5$ ($k_M = - \kappa < 0$), $S$ stands for $\S^2$ and
$\gamma^{(k)}$ is a nonnegative constant determined by the spectrum of the Hodge Laplacian on the $3$-sphere. We thus need to slightly modify
the argument of \cite{EK09} in order to prove that we can still explicitly construct a solution if we add a nonnegative constant to the Hodge Laplacian.

For the convenience of the reader we first give the definitions and some useful properties of the Riesz distributions and the spherical means used in \cite{EK09} and
previously given in \cite{G88,R49}. Let us recall once again that we work on a domain $D \subset M$ which is causal and included in a geodesically normal domain. We
then use the notation $J_D^+(x)$ for the points of $D$ causally connected with $x$, i.e. the points $x' \in D$ such that there exists a future directed, non-spacelike curve
from $x$ to $x'$. We set
\[S_D^+ (x,r) = \{ x' \in J_D^+(x) \, : \, \rho(x,x') = r \},\]
where $\rho : D \times D \rightarrow [0,+\infty)$ is the Lorentzian distance function, and we write $dS$ for the induced area measure on $S_D^+ (x,r)$.

\begin{definition}[\cite{EK09}, Definition 2]\label{defri}
 The Riesz distribution at $x$ with parameter $\alpha$ is the distribution on $D$ defined for $\text{Re}(\alpha) \geq n$ by the integral
 \[ \mathcal{R}_{x,D}^{\alpha}(\phi)(x) := C_{\alpha}  \int_{J_D^+(x)} \rho(x,x')^{\alpha-n} \phi(x') dx',\]
 where $\phi$ is an arbitrary smooth function compactly supported in $D$, $dx'$ stands for the Lorentzian volume element and
 \[ C_{\alpha} := \frac{\pi^{1-\frac{n}{2}} 2^{1-\alpha}}{\Gamma(\frac{\alpha}{2}) \Gamma( \frac{\alpha-n}{2} + 1)}.\]
\end{definition}
\noindent
The map $\alpha \mapsto \mathcal{R}_{x,D}^{\alpha}$ is holomorphic in $\C$ and $\alpha \mapsto \mathcal{R}_{x,D}^{\alpha}(\phi)$ is thus an entire function for any
smooth function $\phi$ with compact support in $D$. Let us recall the following fundamental result on Riesz distributions.

\begin{theorem}[\cite{G88,R49}]\label{OldTheorem1}
 For any fixed $x \in D$, the map $\alpha \mapsto \mathcal{R}_{x,D}^{\alpha}$ can be holomorphically extended to the whole complex plane. Moreover,
 $\Delta_{M} \mathcal{R}_{x,D}^{\alpha} = \mathcal{R}_{x,D}^{\alpha-2}$ and $\mathcal{R}_{x,D}^{0} = \delta_x$.
\end{theorem}

\begin{coro}
 For any $\phi \in C_0^{\infty}(D)$, $u(x) = \mathcal{R}_{x,D}^{2}(\phi)$ solves the hyperbolic equation $\Delta_M u = \phi$ in $D$.
\end{coro}

In order to define the spherical means on $M$,
we now need to introduce bundle-valued distributions $\tau_p$ and $\hat{\tau}_p$ over $D$ using the Lorentzian distance $\rho$ and following \cite{EK09,G88}.
Given a degree $p$, we define distribution-valued double differential $p$-forms $\tau_p$ and $\hat{\tau}_p$ by setting
\[ \tau_0(x,x') := 1, \quad \tau_1(x,x') := -\frac{\sin(\sqrt{k_M}r)}{\sqrt{k_M}} dd' \rho(x,x'),\]
\[ \tau_p(x,x') := \frac{1}{p} \tau_{p-1}(x,x') \wedge \wedge' \tau_1(x,x') \quad \text{for} \quad 2 \leq p \leq n,\]
and
\[ \hat{\tau}_0(x,x') := 0, \quad \hat{\tau}_1(x,x') := - d \rho(x,x') d' \rho(x,x'),\]
\[ \hat{\tau}_p(x,x') := \hat{\tau}_{1}(x,x') \wedge \wedge' \hat{\tau}_{p-1}(x,x') \quad \text{for} \quad 2 \leq p \leq n,\]
at the points $(x,x') \in D \times D$ where the distance function is smooth, and where the unaccented and accented operators $d$ and $d'$ are evaluated at $x$ and $x'$ respectively.
These distributions are thus maps $\tau_p, \, \hat{\tau}_p : D \times D \rightarrow \Lambda^p(T^{\star} M) \boxtimes \Lambda^p(T^{\star} M)$, where
$\Lambda^p(T^{\star} M) \boxtimes \Lambda^p(T^{\star} M)$ is the vector bundle over $D \times D$ whose fiber $(x,x')$ is
$\Lambda^p(T_x^{\star} M) \otimes \Lambda^p(T_{x'}^{\star} M)$.\\
\noindent
The point-wise left and right inner products of a section $\Psi$ of $\Lambda^p(T^{\star} M) \boxtimes \Lambda^p(T^{\star} M)$ with a $p$-form $\omega \in \Omega^p(M)$ are defined by 
\[ \Psi(x,x') \cdot \omega(x) := * \left( \Psi(x,x') \wedge * \omega(x) \right),\]
and
\[ \Psi(x,x') \cdot' \omega(x') := *' \left( \Psi(x,x') \wedge' *' \omega(x') \right).\]
These left and right inner products are thus maps $\Gamma(\Lambda^p(T^{\star}M) \boxtimes \Lambda^p(T^{\star}M)) \times \Omega^p(M) \rightarrow C^{\infty}(M) \otimes \Omega^p(M)$.
Left and right products can likewise be defined for distributions $\Psi : D \times D \rightarrow \Lambda^p(T^{\star}M) \boxtimes \Lambda^p(T^{\star}M)$
and $\omega \in \Omega^p(M)$.\\
\noindent
The double-differential forms we just defined enjoy the notable property that if $x \in D$ and $x' \in J_D^+(x)$ does not lie on the causal cut locus of $x$, then
$\tau_p(x,x') + \hat{\tau}_p(x,x')$ is the parallel transport operator $\Lambda^p(T_x^{\star} D) \rightarrow \Lambda^p(T_{x'}^{\star} D)$ along the unique
future-directed minimal geodesic from $x$ to $x'$. In particular, $\tau_p(x,x') + \hat{\tau}_p(x,x')$ is the identity map on $\Lambda^p(T_x^{\star} M)$.

\begin{definition}[\cite{EK09}, Definition 3]
 Let $\omega \in \Omega_0^p(D)$. Its Lorentzian spherical means of radius $r$ are defined as
  \[ \mathcal{M}_r^D \omega(x) := \frac{(-1)^p}{m(r)} \int_{S_D^+(x,r)} \tau(x,x') \cdot' \omega(x') dS(x'),\]
 and
  \[ \hat{\mathcal{M}}_r^D \omega(x) := \frac{(-1)^p}{m(r)} \int_{S_D^+(x,r)} \hat{\tau}(x,x') \cdot' \omega(x') dS(x'),\]
 where $x \in D$ and $m$ is given by
 \[ m(r) := |\S^{n-1}| \left( \frac{\sin(\sqrt{k_M} r)}{\sqrt{k_M}} \right)^{n-1}.\]
\end{definition}


\begin{prop}[\cite{EK09,G88}]\label{idspherm}
 The spherical means satisfy:
 \[ \Delta_M \left( \int_{0}^{\infty} \left( f(r) \mathcal{M}_r \omega +  \hat{f}(r) \hat{\mathcal{M}}_r \omega \right) d \mu_M(r) \right)
  = \int_{0}^{\infty} \left(F(f,\hat{f})(r) \mathcal{M}_r \omega + \hat{F}(f,\hat{f})(r) \hat{\mathcal{M}}_r \omega \right) d \mu_M(r),\]
 where $d \mu_M(r) := m(r) dr$,
 \[ F(f,\hat{f})(r) := L_M f(r) + pk_M \left( \left( 2 \csc^2(\sqrt{k_M}r) + n - p -1 \right) f(r) - 2 \cos(\sqrt{k_M}r) \csc^2(\sqrt{k_M}r) \hat{f}(r) \right),\]
 and
 \[ \hat{F}(f,\hat{f})(r) := L_M \hat{f}(r) + (n-p)k_M \left( \left( 2 \csc^2(\sqrt{k_M}r) + p -1 \right) \hat{f}(r) - 2 \cos(\sqrt{k_M}r) \csc^2(\sqrt{k_M}r) f(r) \right),\]
with $L_M$ being the radial Laplacian
\[ -L_M := \frac{\partial^2}{\partial r^2} + (n-1) \sqrt{k_M} \cot(\sqrt{k_M}r) \frac{\partial}{\partial r},\]
for all arbitrary smooth functions $f$ and $\hat{f}$ for which the latter integrals converge.
\end{prop}

In order to study the Equations (\ref{inhomo})-(\ref{inhomonous}) we also need some quantities related to the Riemannian surface $S$, still following \cite{EK09}.
We introduce the notation $\mathcal{S}_s$ and $\hat{\mathcal{S}}_s$ for the spherical means of radius $s$ on the Riemannian surface $S$ and call
\[ d \mu_{S} (s) := 2 \pi \frac{\sin(\sqrt{k_S}s)}{\sqrt{k_S}} ds,\]
where $k_S$ is the scalar curvature on $S$, the radial measure on $(0,\text{diam}(S))$. A very useful observation is the following

\begin{lemma}[\cite{EK09}, Lemma 1]
 Let $\beta \in \Omega_0^q(S)$. Then
 \[ \Delta_{S} \left( \int_{0}^{\text{diam}(S)} \omega(s) (\mathcal{S}_s + \hat{\mathcal{S}}_s) \beta d \mu_{S}(s) \right)
  = \int_{0}^{\text{diam}(S)} L_q \omega(s) (\mathcal{S}_s + \hat{\mathcal{S}}_s) \beta d \mu_{S}(s),\]
  where $\omega \in H_{\text{loc}}^2 ((0,\text{diam}(S)),d \mu_{S})$ and
  \[ L_0 = L_2 := - \frac{\partial^2}{\partial s^2} - \sqrt{k_S} \cot(\sqrt{k_S} s) \frac{\partial}{\partial s}\]
  and
  \[ L_1 \omega (s) := L_0 \omega (s) + \frac{2 \sqrt{k_S} \omega (s)}{1 + \cos(\sqrt{k_S} s)}.\]
\end{lemma}
\noindent
The operators $L_q$ defined in the previous lemma are self-adjoint operators on the domains consisting of the functions  
$u \in H^1 ((0,\text{diam}(S)), d\mu_{S})$ such that $L_q u \in L^2((0,\text{diam}(S)),d\mu_{S})$, $\lim\limits_{s \rightarrow 0} s u'(s) = 0$ and
\[ \lim\limits_{s \rightarrow \text{diam}(S)} (\text{diam}(S) - s)u'(s) = 0 \quad \text{if} \quad k_S > 0 \quad \text{and} \quad q = 0,\]
We finally recall the following result given in \cite{EK09} concerning the spectral
decomposition of the operator $L_q$.

\begin{prop}[\cite{EK09}, Proposition 1]\label{specstudy}
 Let $c_{k_S,q} := -\frac{1}{4} k_S$ if $k_S < 0$ and 
 $c_{k_S,q} := 2 k_S \delta_{q 1}$ if $k_S > 0$. Then for each $q= 0, 1, 2$, there exists a Borel measure $\rho_q$ on
 $[c_{k_S,q}, \infty)$ and a $(\mu_{S} \times \rho_q)$-measurable function
 $\omega_q$ defined on $(0,\text{diam}(S)) \times [c_{k_S,q}, \infty)$ such that:
 \begin{enumerate}
  \item $\omega_q(\cdot,\lambda)$ is an analytic formal eigenfunction of $L_q$ with eigenvalue $\lambda$ for $\rho_q$-almost every
  $\lambda \in [c_{k_S,q}, \infty)$.
  \item The map
  \[ u \mapsto \int_{0}^{\text{diam}(S)} u(s) \overline{\omega_q(s,\cdot)} d \mu_{S}(s),\]
  defines a unitary transformation $U_q : L^2((0,\text{diam}(S)),d\mu_{S}) \rightarrow L^2([c_{k_S,q}, \infty), d\rho_q)$ with inverse given by
    \[ U_q^{-1} u := \int_{c_{k_S,q}}^{\infty} \omega_q(\cdot,\lambda) u(\lambda) d \rho_q (\lambda).\]
 \item If $g(L_q)$ is any bounded function of $L_q$, then
     \[ g(L_q) u = \int_{c_{k_S,q}}^{\infty} g(\lambda) \omega_q(\cdot,\lambda) U_q u(\lambda) d \rho_q (\lambda)\]
     in the sense of norm convergence.
 \end{enumerate}
\end{prop}

\begin{remark}
 \begin{enumerate}
  \item The proof of this proposition is given in Appendix $A$ of
  \cite{EK09}. It relies on the Weyl-Kodaira theorem and some properties of hypergeometric functions.
  \item Explicit formulas for all the functions involved in the previous proposition can be found in Lemmas $2$ and $3$ of \cite{EK09}.
 \end{enumerate}
\end{remark}

We note that the decomposition
\[ \Omega^{\star}(M \times S) = \bigoplus_{p,q \geq 0} \Omega^p(M) \otimes \Omega^q(S),\]
defines a natural action of the Laplacians and spherical means operators on $M$ and $S$ on $\Omega^{\star}(M \times S)$. Keeping in mind this idea we now follow Section $5$ of \cite{EK09}. We then analyze the system of ordinary differential equations
\begin{equation}\label{eq1F}
 F(f_{\alpha}(\cdot,\lambda), \hat{f}_{\alpha}(\cdot,\lambda))(r) + \lambda f_{\alpha}(r,\lambda) = C_{\alpha-2} r^{\alpha-n-2},
\end{equation}
and
\begin{equation}\label{eq2F}
 \hat{F}(f_{\alpha}(\cdot,\lambda), \hat{f}_{\alpha}(\cdot,\lambda))(r) + \lambda \hat{f}_{\alpha}(r,\lambda) = C_{\alpha-2} r^{\alpha-n-2},
\end{equation}
where $\lambda$ is positive and we assume that $\text{Re}(\alpha) > n +2$. By using the new variable $z := \sin^2 \left( \frac{\sqrt{k_M} r}{2} \right)$ and introducing
the quantity $l := \frac{\lambda}{k_M}$, Equations (\ref{eq1F})-(\ref{eq2F}) give rise to a single fourth order equation for $f_{\alpha}$, namely
\begin{equation}\label{eqfinalF}
 \frac{d^4 f_{\alpha}}{dz^4} + \sum_{j=0}^3 Q_j(z) \frac{d^j f_{\alpha}}{dz^j} = H_{\alpha}(z),
\end{equation}
where $Q_j$ and $H_{\alpha}$ are explicitly given in \cite{EK09}.

\begin{prop}[\cite{EK09}, Proposition 3]\label{solphi}
Equation (\ref{eqfinalF}) has a unique solution $\Phi_{\alpha,+}(z,l)$ which is continuous in $[0,1]$ if $k_M > 0$ and a unique solution $\Phi_{\alpha,-}(z,l)$ which
is continuous in $(-\infty,0]$ and has the fastest possible decay at infinity if $k_M < 0$. These solutions are real for real $\alpha$.
\end{prop}
\noindent
We now follow the strategy of solution of the inhomogeneous problem given in Section $5$ of \cite{EK09} but keeping in mind that instead of (\ref{inhomo}) we study
the problem (\ref{inhomonous}) coming from (\ref{cauwavpart2}), i.e. we study the operator
\[ \Delta_{M \times \Sigma} + \gamma,\]
for a nonnegative $\gamma$ instead of $\Delta_{M \times \Sigma}$. Let us thus introduce the
$\Lambda^p(T_x^{\star} M)$-valued distribution defined for each $x \in D$ as
\begin{eqnarray*}
 \tilde{\mathcal{R}}_{x,D}^{\alpha}(\phi) &:=& C_{\alpha} \int r^{\alpha-n} \left( \mathcal{M}_r + \hat{\mathcal{M}}_r \right) \phi(x) d \mu_M (r) \\
 &=& C_{\alpha} \int_{J_D^+(x)} \rho(x,x')^{\alpha-n} \left( \tau(x,x') + \hat{\tau}(x,x') \right) \cdot' \phi(x') dx',
\end{eqnarray*}
where $\phi \in \Omega_0^p(M)$. For fixed $x$ and $\lambda$, this distribution is holomorphic for $\text{Re}(\alpha) > n$. The function
$\alpha \mapsto \tilde{\mathcal{R}}_{x,D}^{\alpha}$ has an entire extension and as $\tau(x,x') + \hat{\tau}(x,x') = \text{id}_{\Lambda^p(T_x^{\star}M)}$
we can deduce from Theorem \ref{OldTheorem1} that
\begin{equation}\label{eqtildeR}
 \tilde{\mathcal{R}}_{x,D}^{0} (\phi) = \phi(x).
\end{equation}
We now set
\[ f_{\alpha}(r,\lambda) := \Phi_{\alpha,\text{sign}(k_M)} \left( \sin^2\left(\frac{\sqrt{k_M}r}{2} \right), \frac{\lambda}{k_M} \right),\]
with $\Phi_{\alpha,\pm}$ as in Proposition \ref{solphi} and define $\hat{f}_{\alpha}$ thanks to (\ref{eq1F})-(\ref{eq2F}) in the following way
\[ \hat{f}_{\alpha}(r,\lambda) := \frac{L_M f_{\alpha}(r,\lambda) + pk_M \left( 2 \csc^2 (\sqrt{k_M}r) + n - p - 1 \right) f_{\alpha}(r,\lambda) - C_{\alpha-2}r^{\alpha - n -2}}
 {2 pk_M \cos(\sqrt{k_M}r) \csc^2(\sqrt{k_M}r)}.\]
We then define another vector-valued distribution by
\[ \mathcal{R}_{x,D}^{\alpha,\lambda} (\phi) :=
\int \left( f_{\alpha}(r,\lambda) \mathcal{M}_r + \hat{f}_{\alpha}(r,\lambda) \hat{\mathcal{M}}_r \right) \phi(x) d \mu_M(r), \quad \phi \in \Omega_0^p(M),\]
for $\text{Re}(\alpha) > n-2$. Thanks to Proposition \ref{idspherm}, (\ref{eq1F})-(\ref{eq2F}) and (\ref{eqtildeR}) it follows that the previous distributions satisfy
\[ (\Delta_M + \lambda + \gamma) \mathcal{R}_{x,D}^{\alpha,\lambda+\gamma} = \tilde{\mathcal{R}}_{x,D}^{\alpha-2}.\]
The function $\alpha \mapsto  \mathcal{R}_{x,D}^{\alpha,\lambda+\gamma}$ admits an holomorphic extension to $\mathbb{C}$ (see for instance \cite{BGP07,R49}).
We then have the following proposition:

\begin{prop}[\cite{EK09}, Proposition 4]\label{propsolPois1}
 The distribution $\mathcal{R}_{x,D}^{\alpha,\lambda+\gamma}$ is a holomorphic function of $\alpha \in \C$ for fixed $x \in D$, $\lambda \in \text{spec}(L_q)$
 and $\gamma \geq 0$. Moreover, for any $\phi \in \Omega_0^p(D)$ the differential form $\Phi(x) := \mathcal{R}_{x,D}^{2,\lambda+\gamma}(\phi)$ solves the equation
 \[ (\Delta_M + \lambda + \gamma) \Phi = \phi, \quad \text{in} \quad D.\]
\end{prop}

\begin{remark}
 In \cite{EK09} this result was obtained in the case $\gamma = 0$. Nevertheless, we can still obtain the result as stated in Proposition \ref{propsolPois1}, for nonnegative
 $\gamma$ thanks to the previous analysis. We emphasize that the scalar $\gamma$ is an eigenvalue of the Hodge Laplacian on the $3$-sphere and is 
 thus nonnegative by the structure of the spectrum of the Hodge Laplacian on $\S^3$, see Appendix \ref{apspeclap}. Therefore, the proof of Proposition $3$ of \cite{EK09} applies verbatim
 when replacing $l = \frac{\lambda}{k_M}$ by $l = \frac{ \lambda + \gamma}{k_M}$.
\end{remark}
\noindent
We note that the action of $\mathcal{R}_{x,D}^{\alpha,\lambda+\gamma}$ on $\omega \in \Omega^p(D) \otimes \Omega^q(S)$ defines a bundle-valued distribution on $M \times S$, which acts on the first factor of $\omega$. More precisely, if $\omega(x,y) = \omega_1(x) \otimes
\omega_2(y)$ this action is given by
\[ \mathcal{R}_{x,D}^{\alpha,\lambda+\gamma} (\omega)(y) := \mathcal{R}_{x,D}^{\alpha,\lambda+\gamma}(\omega_1) \otimes \omega_2 (y),\]
with $(x,y) \in D \times S$. Using this result, the solution of (\ref{inhomonous}) is obtained:

\begin{theorem}[\cite{EK09}, Theorem 3]\label{soleqinhomo}
 The differential form
 \[F(x,y) := \int \overline{\omega_q(0,\lambda)} \omega_q(s,\lambda) \mathcal{R}_{x,D}^{2,\lambda+\gamma} (\omega)(y)   d \mu_M (r) d \mu_{S}(s) d\rho_q(\lambda), \]
where the integral ranges over $[c_{k_S,q},\infty) \times (0,\text{diam}(S)) \times (0,\text{diam}(M))$, solves Equation (\ref{inhomonous}) for any compactly
supported $\omega \in \Omega^p(D) \otimes \Omega^q(S)$.
 \end{theorem}
The proof is identical to the one given for Theorem 3 of \cite{EK09}, with $\lambda$ replaced by $\lambda+\gamma$. 

\subsection{Construction of a smooth differential form matching the initial data}\label{secformpart}

The aim of this Section is to construct a $2$-form $\hat{F}_k := \hat{F}_{\beta_k}$ on the causal domain $\tilde{\mathcal{D}} = \mathcal{D}_{| AdS^5 \times \S^2}$
which is contained in a geodesically normal domain of $ AdS^5 \times \S^2$
such that $(\Delta_{\tilde{\mathcal{D}}}^{(k)} + \gamma^{(k)})(\hat{F}_k)$ vanishes at infinite order along the Cauchy surface
$\Sigma_{\tilde{\mathcal{D}}}$ which is the restriction of $\Sigma_{\mathcal{D}}$ on $AdS^5 \times \S^2$.
Such a $2$-form has been already constructed in \cite{BGP07,G88} and we follow here these constructions.

In our analysis, we trivialize the bundle $E$
over $\tilde{\mathcal{D}}$ and identify sections on $E$ with $\C^r$-valued functions where $r$ is the rank of $E$.
We also introduce the function $\beta : \tilde{\mathcal{D}} \rightarrow \R_+^{\star}$ such that the metric takes the form $-\beta dt^2 + g_t$, i.e., by (\ref{metads5}), that is
\[ \beta (x) = \frac{1}{\kappa \sin(x)^2}.\]

We now follow the procedure of \cite{BGP07}, Proposition 3.2.5. Assume first that $F_k$ is a solution to the Cauchy problem (\ref{cauwavpart2}):
\begin{equation}\label{caupropart}
  \left \{
\begin{array}{ccc}
    &(\Delta_{\tilde{\mathcal{D}}}^{(k)} + \gamma^{(k)})F_k = 0& \\
    F_{k \, \,| \Sigma_{\tilde{\mathcal{D}}}} = \mathcal{F}_{0, \, k} & \text{and} & \nabla_{n} F_{k \, \, |\Sigma_{\tilde{\mathcal{D}}}} = \mathcal{F}_{1, \, k}\\
\end{array},
\right.
\end{equation}
of the form
\[ F_k(t,x) = \sum_{j=0}^{\infty} t^j f_{j,k}(x),\]
where $x \in \Sigma_{\tilde{\mathcal{D}}}$. We write
\[ \Delta_{\tilde{\mathcal{D}}}^{(k)} + \gamma^{(k)} = \frac{1}{\beta} \frac{\partial^2}{\partial t^2} + Y,\]
where $Y$ is a differential operator containing $t$-derivatives only up to order $1$. Therefore, the equation
\begin{equation}\label{equ2}
  0 = (\Delta_{\tilde{\mathcal{D}}}^{(k)} + \gamma^{(k)})F_k = \left( \frac{1}{\beta} \frac{\partial^2}{\partial t^2} + Y \right) F_k
 = \frac{1}{\beta} \sum_{j=0}^{\infty} j(j-1)t^{j-2} f_{j,k} + Yu,
\end{equation}
 evaluated at $t=0$ gives
\[ \frac{2}{\beta(x)} f_{2,k}(x) = -Y(\mathcal{F}_{0, \, k}+t \mathcal{F}_{1, \, k})(0,x),\]
for $x \in \Sigma_{\tilde{\mathcal{D}}}$. We see that $f_{2,k}$ is determined by $f_{0,k} := \mathcal{F}_{0, \, k}$ and $f_{1,k} :=\mathcal{F}_{1, \, k}$.
We can then differentiate (\ref{equ2}) with respect to $t$ and repeat the same procedure to prove that each
$f_{j,k}$ can be recursively determined by $f_{0,k}$,...,$f_{j-1,k}$.\\
\noindent
We now do not assume anymore that $F_k$ is a $t$-power series but we still define the $f_{j,k}$, $j \geq 2$, by these recursive relations. Then,
$\text{supp}(f_{j,k}) \subset \text{supp}(\mathcal{F}_{0, \, k}) \cup \text{supp}(\mathcal{F}_{1, \, k})$ for all $j$. Let $\sigma : \R \rightarrow \R$ be a smooth
function such that $\sigma_{| [-1/2,1/2]} \equiv 1$ and $\sigma \equiv 0$ outside $[-1,1]$. We can find a sequence $\epsilon_j \in (0,1)$ such that
\[ \hat{F}_k(t,x) := \sum_{j=0}^{\infty} \sigma \left( \frac{t}{\epsilon_j} \right) t^j f_{j,k}(x),\]
defines a smooth section with compact support that can be differentiating term-wise and by construction one sees that
$\text{supp}(\hat{F}_k) \subset J_{\tilde{\mathcal{D}}}(\text{supp}(\mathcal{F}_{0, \, k}) \cup \text{supp}(\mathcal{F}_{1, \, k}))$
(see \cite{BGP07}).
The form $\hat{F}_k$ is then the form we are interested in since, by the choice of $f_{j,k}$, the section $(\Delta_{\tilde{\mathcal{D}}}^{(k)} + \gamma^{(k)})(\hat{F}_k)$
vanishes to infinite order along $\Sigma_{\tilde{\mathcal{D}}}$.

\subsection{Solution of the sequence of Cauchy problems}\label{secprepar}

The aim of this Section is to use Sections \ref{secpbin} and \ref{secformpart} in order to construct an explicit solution $F_k = F_{\beta_k}$ of the Cauchy problem (\ref{cauwavpart2}), which we recall:
\begin{equation}\label{caupropart2}
  \left \{
\begin{array}{ccc}
    &(\Delta_{\tilde{\mathcal{D}}}^{(k)} + \gamma^{(k)})F_k = 0& \\
    F_{k \, \, | \Sigma_{\tilde{\mathcal{D}}}} = \mathcal{F}_{0, \, k} & \text{and} & \nabla_{n} F_{k \, \, |\Sigma_{\tilde{\mathcal{D}}}} = \mathcal{F}_{1, \, k}\\
\end{array}.
\right.
\end{equation}
To this effect, we follow the procedure of \cite{G88} also used in \cite{BGP07}. First, we recall that, thanks to the results of Section \ref{secformpart}, we know that
$(\Delta_{\tilde{\mathcal{D}}}^{(k)} + \gamma^{(k)})(\hat{F}_k)$ vanishes along $\Sigma_{\tilde{\mathcal{D}}}$ to infinite order. Let us thus set
\begin{equation}\label{defg1}
  G_{+, \, k}(t,x) = \left \{
\begin{array}{ccc}
    (\Delta_{\tilde{\mathcal{D}}}^{(k)} + \gamma^{(k)})(\hat{F}_k) \quad &\text{if}& \quad x \in J_{\tilde{\mathcal{D}}}^+(\Sigma_{\tilde{\mathcal{D}}}) \\
    0 \quad &\text{if}& \quad x \in J_{\tilde{\mathcal{D}}}^-(\Sigma_{\tilde{\mathcal{D}}}) \\
\end{array},
\right.
\end{equation}
which defines a smooth section. We now introduce the inhomogeneous problem
\begin{equation}\label{inhomeq}
(\Delta_{\tilde{\mathcal{D}}}^{(k)} + \gamma^{(k)})(H_{+, \, k}) = G_{+, \, k}.
\end{equation}
Then, using the results of Section \ref{secpbin}, we can explicitly construct a smooth solution
$H_{+, \, k} \in \Omega^2(J_{\tilde{\mathcal{D}}}^+(\Sigma_{\tilde{\mathcal{D}}}))$ of the above problem and this
solution satisfies $\text{supp}(H_{+,\, k}) \subset \text{supp}(G_{+, \, k}) \subset J_{\tilde{\mathcal{D}}}^+(\Sigma_{\tilde{\mathcal{D}}})$.\\
\noindent
We now consider the ``future'' Cauchy problem on $\Omega^2(J_{\tilde{\mathcal{D}}}^+(\Sigma_{\tilde{\mathcal{D}}}))$:
\begin{equation}\label{futcau}
  \left \{
\begin{array}{ccc}
    &(\Delta_{\tilde{\mathcal{D}}}^{(k)} + \gamma^{(k)})F_{+, \, k} = 0& \\
    F_{+, \, k \, \, | \Sigma_{\tilde{\mathcal{D}}}} = \mathcal{F}_{0, \, k} & \text{and} & \nabla_{n} F_{+, \, k \, \,|\Sigma_{\tilde{\mathcal{D}}}} = \mathcal{F}_{1, \, k}\\
\end{array}, \quad F_{+, \, k} \in \Omega^2(J_{\tilde{\mathcal{D}}}^+(\Sigma_{\tilde{\mathcal{D}}})).
\right.
\end{equation}
We set
\[ F_{+, \, k} = \hat{F}_k - H_{+, \, k},\]
where $\hat{F}_k$ was explicitly constructed in Section \ref{secformpart} and $H_{+, \, k}$ was also explicitly constructed in Section \ref{secpbin}. The function $F_{+, \, k}$ is smooth and 
since $H_{+, \, k} = 0$ on $J_{\tilde{\mathcal{D}}}^-(\Sigma_{\tilde{\mathcal{D}}})$, $F_{+, \, k} = \hat{F}_k$ to infinite order on $\Sigma_{\tilde{\mathcal{D}}}$.
In particular,
\[     F_{+, \, k \, \, | \Sigma_{\tilde{\mathcal{D}}}} = \mathcal{F}_{0, \, k}  \quad \text{and}  \quad \nabla_{n} F_{+, \, k \, \,|\Sigma_{\tilde{\mathcal{D}}}} = \mathcal{F}_{1, \, k}.\]
Moreover,
\begin{eqnarray*}
 (\Delta_{\tilde{\mathcal{D}}}^{(k)} + \gamma^{(k)})F_{+, \, k} &=& (\Delta_{\tilde{\mathcal{D}}}^{(k)} + \gamma^{(k)})(\hat{F}_k-H_{+, \, k}) \\
 &=& (\Delta_{\tilde{\mathcal{D}}}^{(k)} + \gamma^{(k)})(\hat{F}_k) - (\Delta_{\tilde{\mathcal{D}}}^{(k)} + \gamma^{(k)})(H_{+, \, k}) \\
 &=& G_{+, \, k} - G_{+, \, k} \\
 &=& 0, \quad \text{on} \quad J_{\tilde{\mathcal{D}}}^+(\Sigma_{\tilde{\mathcal{D}}}).
\end{eqnarray*}
Therefore, $F_{+, \, k}$ is a solution of the ``future" Cauchy problem (\ref{futcau}).\\
\noindent
Similarly, by considering the form
\begin{equation}\label{defg2}
  G_{-, \, k}(t,x) = \left \{
\begin{array}{ccc}
    0 \quad &\text{if}& \quad x \in J_{\tilde{\mathcal{D}}}^+(\Sigma_{\tilde{\mathcal{D}}}) \\
    (\Delta_{\tilde{\mathcal{D}}}^{(k)} + \gamma^{(k)})(\hat{F}_k) \quad &\text{if}& \quad x \in J_{\tilde{\mathcal{D}}}^-(\Sigma_{\tilde{\mathcal{D}}}) \\
\end{array},
\right.
\end{equation}
we then obtain a solution of the ``past'' Cauchy problem:
\begin{equation}\label{pascau}
  \left \{
\begin{array}{ccc}
    &(\Delta_{\tilde{\mathcal{D}}}^{(k)} + \gamma^{(k)})F_{-, \, k} = 0& \\
    F_{-, \, k \, \,| \Sigma_{\tilde{\mathcal{D}}}} = \mathcal{F}_{0, \, k} & \text{and} & \nabla_{n} F_{-, \, k \, \,|\Sigma_{\tilde{\mathcal{D}}}} = \mathcal{F}_{1, \, k}\\
\end{array}, \quad F_{-, \, k} \in \Omega^2(J_{\tilde{\mathcal{D}}}^-(\Sigma_{\tilde{\mathcal{D}}})),
\right.
\end{equation}
by setting
\[ F_{-, \, k}(t,x) = \hat{F}_k(t,x) - H_{-, \, k}(t,x),\]
where $H_{-, \, k}$ is the solution of the inhomogeneous problem
\[  (\Delta_{\tilde{\mathcal{D}}}^{(k)} + \gamma^{(k)})(H_{-, \, k} ) = G_{-, \, k},\]
constructed in Section \ref{secpbin}.\\
\noindent
Finally, we put
\begin{equation}\label{defFk}
  F_k(t,x) = \left \{
\begin{array}{ccc}
    F_{+, \, k} (t,x) &\text{if}& \quad x \in J_{\tilde{\mathcal{D}}}^+(\Sigma_{\tilde{\mathcal{D}}}) \\
    F_{-, \, k} (t,x) &\text{if}& \quad x \in J_{\tilde{\mathcal{D}}}^-(\Sigma_{\tilde{\mathcal{D}}}) \\
\end{array}.
\right.
\end{equation}
Then, since $F_{+, \, k}$ and $F_{-, \, k}$ coincide to infinite order on $\Sigma_{\tilde{\mathcal{D}}}$, we conclude that $F_k$ is a smooth $2$-form on
$J_{\tilde{\mathcal{D}}}(\Sigma_{\tilde{\mathcal{D}}})$,
and in particular that 
\[ F_{k \, \, | \Sigma_{\tilde{\mathcal{D}}}} = \mathcal{F}_{0, \, k} \quad \text{and} \quad \nabla_{n} F_{k \, \, |\Sigma_{\tilde{\mathcal{D}}}} = \mathcal{F}_{1, \, k}.\]
Moreover, by construction
\[ (\Delta_{\tilde{\mathcal{D}}}^{(k)} + \gamma^{(k)})F_k = 0.\]
The section $F_{\beta_k} = F_k \in \Omega^2(J_{\tilde{\mathcal{D}}}(\Sigma_{\tilde{\mathcal{D}}}))$, explicitly constructed above, is thus a solution of the
Cauchy problem (\ref{caupropart2}).

\subsection{Solution of the full Cauchy problem}\label{secfinpre}

In the previous section we have explicitly constructed in (\ref{defFk}) for each $k \in \{ 0, 1, 2\}$ a solution $F_k = F_{\beta_k}$ of the (\ref{caupropart2})
Cauchy problems restricted to a Fourier mode of  the Hodge Laplacian on the $3$-sphere:
\begin{equation*}
  \left \{
\begin{array}{ccc}
    &(\Delta_{\tilde{\mathcal{D}}}^{(k)} + \gamma^{(k)})F_k = 0& \\
    F_{k \, \,| \Sigma_{\tilde{\mathcal{D}}}} = \mathcal{F}_{0, \, k} & \text{and} & \nabla_{n} F_{k \, \, |\Sigma_{\tilde{\mathcal{D}}}} = \mathcal{F}_{1, \, k}\\
\end{array},
\right.
\end{equation*}
which is a concise formulation of the sequence of Cauchy problems:
\begin{equation*}
 \left \{
\begin{array}{ccc}
    &(\Delta_{\tilde{\mathcal{D}}}^{(2-k)} + \lambda_{\beta_k}^{(k)})F_{\beta_k} = 0& \\
    F_{\beta_k \, \, | \Sigma_{\tilde{\mathcal{D}}}} = \mathcal{F}_{0, \, \beta_k} & \text{and} & \nabla_{n} F_{\beta_k \, \, |\Sigma_{\tilde{\mathcal{D}}}} = \mathcal{F}_{1, \, \beta_k}\\
\end{array},
\right.
\end{equation*}
where $F_{\beta_k} \in \Omega^{2-k}(AdS^5 \times \S^2)$ is defined on $\tilde{\mathcal{D}}$, $\Sigma_{\tilde{\mathcal{D}}}$ is the restriction of $\Sigma_{\mathcal{D}}$ on
$AdS^5 \times \S^2$ and $\mathcal{F}_{0, \, \beta_k}$ and $\mathcal{F}_{1, \, \beta_k}$ are the corresponding initial data.
In order to obtain a solution of the complete Cauchy problem
 \begin{equation*}
 \left \{
\begin{array}{ccc}
    &\Delta_{\mathcal{D}}^{(2)} (F) = 0& \\
    F_{| \Sigma} = \mathcal{F}_0 & \text{and} & \nabla_{n} F_{|\Sigma} = \mathcal{F}_1\\
\end{array},
\right.
\end{equation*}
we then only need to add the preceding solutions in the following way:
\begin{equation}\label{construfinalF}
 F = \sum_{k=0}^2 \sum_{\beta_k \in B_k} F_{\beta_k} \wedge \eta_{\beta_k}^{(k)}.
\end{equation}
In order to complete our work we need to check the convergence of the series (\ref{construfinalF}), the smoothness the the $2$-form defined as its of its sum and check that this $2$-form is a solution of the source-free Maxwell equations. The proof of convergence of the Fourier expansion in terms of the eigenforms of the Hodge Laplacian on $\S^3$ is straightforward and uses the Bessel inequality as in the standard scalar case.  Concerning the smoothness of the $2$-form $F$ defined in (\ref{construfinalF}) this is a consequence of the fact that the solutions $F_{\beta_k}$ of the sequence of Cauchy problems are smooth and the eigenforms of the Hodge Laplacian on the $3$-sphere $\eta_{\beta_k}^{(k)}$ are also smooth. Finally, $F$ is, by construction, a solution of the Maxwell equations since it is a solution of the Cauchy problem for the Hodge Laplacian introduced in Section \ref{secredwav}. Indeed, using the decomposition of the Hodge Laplacian studied in Section \ref{secdechod} and the
construction in Section \ref{secprepar} of the solutions $F_{\beta_k}$ of the sequence of Cauchy problems, we immediately obtain that $F$ is a solution of the Cauchy problem for the Hodge Laplacian. 

In conclusion, the $2$-form $F \in \Omega^2(\mathcal{D})$ defined by (\ref{construfinalF}) is an explicit solution of Cauchy problem for the source-free Maxwell equations in a causal domain $\cal D$ contained in a geodesically normal domain of the Lorentzian manifold $AdS^5 \times \S^2 \times \S^3$, expressed in terms of the forms $F_{\beta_k}$ constructed above in (\ref{defFk}) and the eigenforms $\eta_{\beta_k}^{(k)}$ of the Hodge Laplacian on the $3$-sphere, which are explicitly recalled in Appendix \ref{apspeclap}. It thus provides an explicit representation of the solution of the Green's function.

\appendix

\section{Spectral theory of the Hodge Laplacian on the 3-sphere}\label{apspeclap}

The aim of this Appendix is to recall some results on the spectral theory of the Hodge Laplacian on the $3$-sphere. In particular we are interested in the spectrum
and explicit orthonormal eigenmodes of $\Delta_{\S^3}^{(k)}$ for $k \in \{0,1,2,3\}$ with respect to the Hodge scalar product.
This question has been largely studied and we will, in this Appendix, simply and strictly recall the results
contained in the well-written and self-contained paper \cite{BHQR16}. We also refer the reader to the references \cite{A73,CT85,GS78,H91,J78,L04,RO84,RO85} given in this last paper.

We first study the scalar harmonics on the unit sphere $\S^3$. We introduce the H\"opf coordinates 
\begin{equation*}
 \left \{
\begin{array}{cc}
    x^1 &= \sin(\alpha) \cos(\varphi) \\
    x^2 &= \sin(\alpha) \sin(\varphi)\\
    x^3 &= \cos(\alpha) \cos(\theta)\\
    x^4 &= \cos(\alpha) \sin(\theta)\\
\end{array},
\right.
\end{equation*}
where $\alpha \in [0,\pi/2]$, $(\theta,\varphi) \in [ 0, 2\pi [^2$ and we study the Laplace-de Rham operator $\Delta_{\S^3}^{(k)} := -(d \delta + \delta d)$ acting on the
scalar functions on $\S^3$, i.e. in the case $k=0$. Let us thus consider the normalized scalar modes for this operator, i.e. the solutions of the eigenvalue equation
\[ \Delta_{\S^3}^{(0)} \Phi_i = \lambda_i \Phi_i,\]
where $\Phi_i$ are the modes corresponding to the eigenvalue $\lambda_i = -L(L+2)$, $L \in \N$. The index $i = (L,m_+,m_-)$ allows us to label the modes in the following
way: $|m_{\pm}| \leq \frac{L}{2}$ and $\frac{L}{2} - m_{\pm} \in \N$, and
\[ \Phi_i = T_{L,m_+,m_-}(\alpha,\varphi,\theta) := C_{L,m_+,m_-} e^{i(S \varphi + D \theta)}(1-x)^{\frac{S}{2}} (1+x)^{\frac{D}{2}} P_{\frac{L}{2} - m_+}^{(S,D)}(x),\]
where $P_n^{(a,b)}$ is a Jacobi polynomial, $x = \cos(2\alpha)$, $S := m_+ + m_-$, $D := m_+ - m_-$ and
\[ C_{L,m_+,m_-} := \frac{1}{2^m + \pi} \sqrt{\frac{L+1}{2}} \sqrt{\frac{(L/2 + m_+)! (L/2 - m_+)!}{(L/2 + m_-)! (L/2 - m_-)!}} .\]

Before studying the eigenforms of $\Delta_{\S^3}^{(1)}$ we need to introduce some extra quantities. To this effect, let us introduce the two Killing vectors
\[ \xi := X_{12} + X_{34} = \partial_{\varphi} + \partial_{\theta},\]
\[ \xi' := X_{12} - X_{34} = \partial_{\varphi} - \partial_{\theta},\]
where $X_{ij} := x_i \partial_j - x_j \partial_i$ are the generators of the $so(4)$ algebra.\\
\noindent
Let us now study the eigenvalue equation
\[ \Delta_{\S^3}^{(1)} \alpha = \lambda \alpha,\]
where $\alpha$ is a one-form on $\S^3$. Using the Hodge decomposition we can prove that the space of solutions of this equation is the direct sum of two orthogonal
subspaces containing the exact and the co-exact one-forms, respectively. Concerning the exact one-forms, the eigenvalues are $-L(L+2)$, $L \in \N \setminus \{0\}$, the
eigenforms are given by the exterior derivatives of the scalar modes and the dimension of the associated proper subspace $\varepsilon_L^E$ is $d^E = (L+1)^2$. The 
eigenvalues for the co-exact one-forms are given by $-L^2$, $L \in \N \setminus \{0,1\}$ and the dimension of the associated proper subspaces 
$\varepsilon_L^{CE}$ is $d^{CE} = 2(L-1)(L+1)$.\\
\noindent
We now construct an explicit orthonormal basis of eigenforms. For this purpose let us define:
\[ A_i := d \Phi_i,\]
\[ B_i := * d \Phi_i \tilde{\xi}, \quad B_i' := * d \Phi_i \tilde{\xi'},\]
\[ C_i := * d B_i, \quad C_i' := * d B_i',\]
where $\tilde{\xi}$ is the one-form associated with the vector $\xi$, i.e. $\tilde{\xi} := \flat \xi$. In these expressions $A_i$ is an exact one-form whereas $B_i$, $C_i$,
$B_i'$ and $C_i'$ are co-exact one-forms. Finally, we introduce the combinations
\[ E_i := (L+2) B_i + C_i \quad \text{and} \quad E_i' := (L+2) B_i' - C_i',\]
where $i = (L,m_+,m_-)$ (respectively $i = (L,m_+',m_-')$). The main result of \cite{BHQR16} is then the following:

\begin{theorem}[\cite{BHQR16}]
 \begin{enumerate}
  \item The one-forms $E_{L,m_+,m_-}$ and $E_{L,m_+',m_-'}'$ satisfy
  \[ \Delta E_{L,m_+,m_-} = -L^2 E_{L,m_+,m_-},\]
  \[ \Delta E_{L,m_+',m_-'}'= -L^2 E_{L,m_+',m_-'}',\]
  for $L \geq 2$.
  \item The family of one-forms
  \begin{equation*}
 \left \{
\begin{array}{c}
    E_{L,m_+,m_-}, \quad L \geq 2, \quad |m_+| \leq \frac{L}{2} - 1, \quad |m_-| \leq \frac{L}{2},\\
    E_{L,m_+',m_-'}', \quad L \geq 2, \quad |m_-'| \leq \frac{L}{2} - 1, \quad |m_+'| \leq \frac{L}{2},\\
\end{array}
\right.
\end{equation*}
once normalized, form an orthonormal basis of the corresponding proper subspace of co-exact one-forms $\varepsilon_L^{CE}$.
\item The whole set of one-forms: exact $\{ A_{L,m_+,m_-} \}$, for $L \geq 1$, and co-exact $\{ E_{L,m_+,m_-}, E_{L,m_+',m_-'}' \}$, as above, form a complete orthonormal
set of modes for the Laplace-de Rham operator on $\S^3$.
 \end{enumerate}
\end{theorem}

\begin{remark}
 Using the Hodge star operator we also obtain explicit orthonormal basis of eigenforms for $\Delta_{\S^3}^{(2)}$ and $\Delta_{\S^3}^{(3)}$ from the previous ones.
\end{remark}

\section{Hodge decomposition and perspectives}\label{appersp}

The aim of this Appendix is to give some remarks and ideas on the Cauchy problem we study in this paper in order to suggest another way to solve it and also to propose some
generalizations of our result, which was initially motivated by the Cauchy problem studied in \cite{EK10}.

As we did before in our paper we study the Cauchy problem for the Hodge Laplacian given by
\begin{equation*}
 \left \{
\begin{array}{ccc}
    &\Delta_{AdS^5 \times \S^2 \times \S^3}^{(2)} (F) = 0& \\
    F_{| \Sigma} = \mathcal{F}_0 & \text{and} & \nabla_{n} F_{|\Sigma} = \mathcal{F}_1\\
\end{array},
\right.
\end{equation*}
where $\Delta^{(2)} = -(d \delta + \delta d)$ is the Hodge Laplacian acting on $2$-forms of $AdS^5 \times \S^2 \times \S^3$
and $\Sigma$ is a Cauchy hypersurface of $AdS^5 \times \S^2 \times \S^3$.
Since $\mathcal{M} = AdS^5 \times \S^2 \times \S^3$ is a product manifold we can use a decomposition of the Hodge Laplacian which is different from the one we used in Section
\ref{secdechod} that allows us to obtain equations only on $AdS^5$. Precisely, we know that
\[ \Delta^{(2)}_{AdS^5 \times \mathbb{S}^2 \times \mathbb{S}^3} = \begin{pmatrix}
\Delta^{(0)}_{AdS^5} \oplus \Delta^{(2)}_{\S^2 \times \mathbb{S}^3} & 0 & 0 \\
0 & \Delta^{(1)}_{AdS^5} \oplus \Delta^{(1)}_{\S^2 \times \mathbb{S}^3} & 0 \\
0 & 0 &\Delta^{(2)}_{AdS^5} \oplus \Delta^{(0)}_{\S^2 \times \mathbb{S}^3}
\end{pmatrix}.\]
We can then use the explicit eigenmodes on the $2$ and $3$-spheres given for instance in \cite{BHQR16,KMHR00,L04} in order to simplify the part corresponding to the 
Hodge Laplacian on the product of spheres $\S^2 \times \S^3$. Using a well-chosen basis of eigenforms on $\S^2 \times \S^3$ we obtain
a new countable family of Cauchy problems on each mode, given, for $k \in \{0,1,2 \}$, by:
 \begin{equation*}
 \left \{
\begin{array}{ccc}
    &(\Delta_{AdS^5}^{(k)} + \lambda_{\beta_k}^{(2-k)})F_{\beta_k} = 0& \\
    F_{\beta_k \, \, | \tilde{\Sigma}} = \mathcal{F}_{0, \, \beta_k} & \text{and} & \nabla_{n} F_{\beta_k \, \, |\tilde{\Sigma}} = \mathcal{F}_{1, \, \beta_k}\\
\end{array},
\right.
\end{equation*}
where $(\lambda_{\beta_k}^{(k)})_{\beta_k}$ is the spectrum of the Hodge $k$-Laplacian on $\S^2 \times \S^3$, $\Delta_{\S^2 \times \S^3}^{(k)}$, $\tilde{\Sigma}$ is the
restriction of $\Sigma$ on $AdS^5$ and $F_{\beta_k} \in \Omega^2(AdS^5)$.\\
\noindent
First of all, let us note that the scalar block $(\Delta_{AdS^5}^{(0)} + \gamma^{(0)})F_{\beta_0}$, where $\gamma^{(k)} = \lambda_{\beta_{2-k}}^{(2-k)}$, has been studied in
\cite{IW03} and that we can therefore use this latter work to obtain an explicit representation of $F_{\beta_0}$. Indeed, in \cite{EK10} the authors obtain a scalar ODE in terms of $(t,x)$
that falls within the framework of the important contribution \cite{IW03} where a representation of the solution of an equation of the form
\begin{equation}\label{eqIW}
 \partial_{tt}^2 \Phi = -A \Phi, 
\end{equation}
for a positive and symmetric operator $A$ in the variable $x$, is given.\\
\noindent
Regarding the other blocks, we could use a moving frame and the Hodge decomposition on differential forms to eliminate the $\S^3$-contribution from the $AdS^5$ geometry (see (\ref{metads5})) and then obtain an
operator in terms of $(t,x) \in (N^2,g_{N^2})$, where $g_{N^2} = \frac{-dt^2 + dx^2}{\kappa \sin(x)^2}$. More precisely, on this $2$-manifold $N^2$ we obtain from the previous
equations on the Hodge k-Laplacians of $AdS^5$ two kinds of equations given by:
\[ (\Delta_{N^2}^{(0)} + \gamma^{(0)})F^{(0)}(t,x) = 0,\]
which is of the previous form and can be explicitly solved by \cite{IW03}, and
\begin{equation}\label{eqmat}
 (\Delta_{N^2}^{(1)} + \gamma^{(1)})F^{(1)}(t,x) = 0,
\end{equation}
which is slightly different. Indeed, there are two differences between (\ref{eqmat}) and Equation (\ref{eqIW}) of \cite{IW03}. First, Equation (\ref{eqmat}) is a matrix equation but this point should not be an obstacle. More important is the fact that
Equation (\ref{eqmat}) contains a matrix-valued term of order one in time that a priori prevents us from using \cite{IW03}.
This last term is the reason why we did not use directly \cite{IW03} in our work. Nevertheless, we could probably still use a similar analysis to obtain a
representation of the solution of the Maxwell equations in our framework in the spirit of \cite{IW03}.

We would also like to mention the work \cite{IW04} in which the Maxwell
equations are shown to be of the form (\ref{eqIW}) on globally hyperbolic manifolds using potential forms. The reason why we did not use this work lies 
in the fact that, even if we restricted our analysis to a geodesically normal domain, we do not want to introduce such potential forms because of possible extensions to the entire non-globally hyperbolic manifold $\mathcal{M} = AdS^5 \times \S^2 \times \S^3$.

Finally, for future work, we are interested in obtaining a representation of the solution of the Maxwell equations on $AdS^5 \times \S^2 \times \S^3$ similar to the one given in \cite{IW03}
because we would like
to generalize this result to the $AdS^5 \times Y^{p,q}$ solutions of type IIB supergravity discovered in \cite{GMSW}, where $Y^{p,q}$ is a certain irregular cohomogeneity one Sasaki-Einstein structure on $\S^2\times \S^3$, for which the representation of the solution of the Klein-Gordon equation in terms of a Green's function is given in \cite{EK10}. If we could use the procedure coming from \cite{EK10,IW03} on $AdS^5 \times \S^2 \times \S^3$ we would then need to obtain a spectral representation of the Hodge Laplacian on $Y^{p,q}$ in order to adapt our result on the manifold $AdS^5 \times Y^{p,q}$. This question is being currently investigated.

\vspace{0,5cm}
\noindent
\textit{Acknowledgments:} Research supported by NSERC RGPIN 105490-2011.

{}

\noindent
D\'epartement de Math\'ematiques, Universit\'e de Nantes, 2, rue de la Houssini\`ere, BP 92208, 44322 Nantes Cedex 03, France.\\
\textit{Email adress}: damien.gobin@univ-nantes.fr.\\

\noindent
Department of Mathematics and Statistics, McGill University, Montreal, QC, H3A 0B9, Canada.\\
\textit{Email adress}: niky.kamran@mcgill.ca.


\end{document}